\documentclass[11pt,reqno]{amsart}

\usepackage[T1]{fontenc}
\usepackage[utf8]{inputenc}
\usepackage{lmodern}
\usepackage{amsmath,amssymb,amsthm,mathtools}
\usepackage{microtype}
\usepackage{enumitem}
\usepackage[hidelinks]{hyperref}
\hypersetup{
  pdftitle={Monomial bases for holomorphic functions on Banach spaces with an unconditional basis},
  pdfauthor={Thiago Grando},
  pdfsubject={Monomial Schauder bases in spaces of holomorphic functions},
  pdfkeywords={holomorphic functions, Schauder basis, monomial basis, unconditional basis}
}

\setlist[enumerate]{label=(\roman*),leftmargin=2.2em}
\setlist[itemize]{leftmargin=2em}

\newtheorem{theorem}{Theorem}[section]
\newtheorem{proposition}[theorem]{Proposition}
\newtheorem{lemma}[theorem]{Lemma}
\newtheorem{corollary}[theorem]{Corollary}
\theoremstyle{definition}

\newtheorem{example}[theorem]{Example}
\theoremstyle{remark}
\newtheorem{remark}[theorem]{Remark}

\newcommand{\C}{\mathbb C}
\newcommand{\N}{\mathbb N}
\newcommand{\Nzero}{\mathbb N_0}
\newcommand{\calH}{\mathcal H}
\newcommand{\calP}{\mathcal P}
\newcommand{\supp}{\operatorname{supp}}

\newcommand{\id}{\operatorname{Id}}
\newcommand{\e}{\mathrm e}
\newcommand{\abs}[1]{\lvert #1\rvert}
\newcommand{\norm}[1]{\lVert #1\rVert}

\newcommand{\ellength}{\operatorname{len}}
\newcommand{\dd}{\,\mathrm d}

\title[Monomial bases for holomorphic functions]
{Monomial bases for holomorphic functions on Banach spaces with an unconditional basis}

\author[T. Grando]{Thiago Grando}
\address{Department of Mathematics, Midwestern Paran\'a State University, Guarapuava, Brazil}
\email{tgrando@unicentro.br}

\date{}

\begin{document}

\begin{abstract}
Let $X$ be a complex Banach space with an unconditional Schauder basis. We prove that the monomials associated with this basis, endowed in each homogeneous degree with the square order and globally with any compatible ordering, form a Schauder basis for the space $(\calH(X),\tau_0)$ of entire holomorphic functions on $X$ with the compact-open topology. The proof relies on two main ideas. First, coordinate-tail conditions yield a fundamental system of compact solid subsets of $X$. Second, Fourier projections in the coordinates give the uniform estimate $c_{K,n}\leq n+1$ for the basis constant of the degree-$n$ monomials with respect to the supremum seminorm on each such compact set $K$. We derive applications to Banach sequence lattices with dense $c_{00}$, including classical, Lorentz, Orlicz-heart and Schreier-type sequence spaces, and to general vector-valued $E$-sums. In particular, the result covers mixed $c_0$- and $\ell_r$-sums of finite-dimensional $\ell_p$ spaces and the Lorentz predual $d_*(w,1)$.
\end{abstract}

\maketitle

\section{Introduction}

Let $X$ be a complex Banach space with a Schauder basis. A classical problem in infinite-dimensional holomorphy asks when the monomials associated with the basis of $X$ form a Schauder basis in spaces of homogeneous polynomials or holomorphic functions on $X$. Ryan proved that, for every $n\in\N$, the degree-$n$ monomials in the square order form a Schauder basis for $(\calP({}^nX),\tau_0)$, where $\tau_0$ denotes the compact-open topology; see \cite{RyanThesis}. Under additional hypotheses, analogous results hold for normed spaces of homogeneous polynomials. Dimant and Dineen obtained a related result for polynomials weakly continuous on bounded sets \cite{DimantDineen}, while Defant and Kalton investigated unconditionality in spaces of homogeneous polynomials \cite{DefantKalton}.

For spaces of entire functions, the passage from bases in the homogeneous components to a single basis of the whole space requires quantitative control of the basis constants as the degree grows. Dineen and Mujica established an abstract criterion in terms of $S^*$-absolute decompositions \cite[Theorem~1]{DineenMujica} and applied it to $(\calH(c_0),\tau_\omega)$ and $(\calH_b(c_0),\tau_b)$. Ryan had previously obtained an equicontinuous unconditional monomial basis for $(\calH(\ell_1),\tau_0)$ \cite{RyanL1}. More recently, the monomial basis problem for $(\calH(\ell_p),\tau_0)$ was considered in \cite{GrandoLourencoLp}. Related work of Dineen concerns holomorphic decompositions and the coincidence of natural locally convex topologies on spaces of holomorphic functions over spaces with bases or unconditional decompositions; see \cite{Dineen1995Studia,Dineen1995Extracta}. The question addressed here is quantitative and different in emphasis: we seek a global monomial Schauder basis for the compact-open topology and obtain an explicit degree-wise bound for the corresponding initial projections.

The purpose of this paper is to give a common argument that applies whenever the underlying Banach space has an unconditional Schauder basis. After an equivalent renorming, such a basis may be assumed to be $1$-unconditional. If $(e_j)$ denotes this basis and $\pi_N$ the $N$th coordinate projection, we introduce compact sets of the form
\[
 K_\eta=\bigl\{x\in X:\norm{(\id-\pi_N)x}\leq \eta_N\text{ for every }N\in\Nzero\bigr\},
\]
where $\eta=(\eta_N)$ is a positive null sequence and $\pi_0=0$. These sets form a fundamental system of compact subsets of $X$ and are solid with respect to the basis.

The decisive estimate is obtained by coordinatewise Fourier projections. If $K$ is a bounded solid set and $c_{K,n}$ is the basis constant of the degree-$n$ monomials in the square order for the seminorm of uniform convergence on $K$, then
\[
 c_{K,n}\leq n+1.
\]
In particular, $\limsup_{n\to\infty}c_{K,n}^{1/n}=1$. The criterion of Dineen and Mujica then yields the desired basis of $(\calH(X),\tau_0)$. This estimate replaces arguments based on comparing the norm of a product with the product of two norms; such comparisons need not admit a constant independent of the degree.

The general theorem yields a broad collection of applications. In particular, it applies to Banach sequence lattices for which $c_{00}$ is dense, including $c_0$, $\ell_p$, Lorentz sequence spaces, Orlicz hearts and Schreier-type spaces. It is also stable under vector-valued $E$-sums of spaces with $1$-unconditional bases. Thus, for example, it applies to
\[
 c_0\left(\bigoplus_{j=1}^{\infty}\ell_{p_j}^{d_j}\right)
 \quad\text{and}\quad
 \ell_r\left(\bigoplus_{j=1}^{\infty}\ell_{p_j}^{d_j}\right),
\]
where $1\leq r<\infty$, $1\leq p_j<\infty$ and $d_j\in\N$. The original examples
\[
 c_0\left(\bigoplus_{j=1}^{\infty}\ell_p^j\right)
 \quad\text{and}\quad d_*(w,1)
\]
are recovered as concrete motivating cases.

The paper is organized as follows. Section~\ref{sec:preliminaries} recalls the square order and the Dineen--Mujica criterion. Section~\ref{sec:compact} constructs the compact solid sets. Section~\ref{sec:fourier} proves the estimate $c_{K,n}\leq n+1$. Section~\ref{sec:main} proves the main theorem. Section~\ref{sec:applications} develops general applications to sequence lattices, vector-valued $E$-sums and the motivating concrete examples.

\section{Preliminaries}\label{sec:preliminaries}

All vector spaces are over $\C$. For a Banach space $X$, we write $\calP({}^nX)$ for the space of continuous scalar-valued $n$-homogeneous polynomials on $X$, with the convention $\calP({}^0X)=\C$. Every $f\in\calH(X)$ has a Taylor expansion at the origin,
\[
 f=\sum_{n=0}^{\infty}P_n,
\]
where $P_n\in\calP({}^nX)$ and the series converges uniformly on compact subsets of $X$. The compact-open topology $\tau_0$ on $\calH(X)$ is generated by the seminorms
\[
 \norm{f}_K=\sup_{x\in K}\abs{f(x)},
\]
where $K$ ranges over the compact subsets of $X$.

Let $(e_j)_{j=1}^{\infty}$ be a Schauder basis of $X$, with coefficient functionals $(e_j^*)_{j=1}^{\infty}$ and coordinate projections
\[
 \pi_Nx=\sum_{j=1}^{N}e_j^*(x)e_j,\qquad N\in\N,
 \quad\text{and}\quad \pi_0=0.
\]
We use the set
\[
 \Nzero^{(\N)}=\bigl\{\alpha=(\alpha_j)_{j\geq1}:\alpha_j\in\Nzero
 \text{ and }\alpha_j=0\text{ for all but finitely many }j\bigr\}
\]
of finitely supported multi-indices. For $\alpha\in\Nzero^{(\N)}$, set
\[
 \abs{\alpha}=\sum_{j=1}^{\infty}\alpha_j,
 \qquad
 \ellength(\alpha)=\max\supp(\alpha),
\]
with $\ellength(0)=0$, and define the associated monomial by
\[
 e^{*\alpha}(x)=\prod_{j=1}^{\infty}e_j^*(x)^{\alpha_j}.
\]

For multi-indices $\alpha$ and $\beta$ of the same degree, the \emph{square order} is defined by $\alpha\prec\beta$ if either
\begin{enumerate}
 \item $\ellength(\alpha)<\ellength(\beta)$, or
 \item $\ellength(\alpha)=\ellength(\beta)$ and, for
 \[
  j=\max\{i:\alpha_i\neq\beta_i\},
 \]
 one has $\alpha_j<\beta_j$.
\end{enumerate}
This is the reverse lexicographic order within each fixed degree. We denote the degree-$n$ monomials in this order by $(P_{n,m})_{m\geq1}$.

An ordering of all monomials is called \emph{compatible} if its restriction to the degree-$n$ monomials agrees with the square order for every $n$.

We recall the criterion used to assemble the bases of the homogeneous components. A sequence $(E_n)_{n\geq0}$ of subspaces of a locally convex space $E$ is a \emph{decomposition} if every $x\in E$ admits a unique representation $x=\sum_{n=0}^{\infty}x_n$ with $x_n\in E_n$. It is called \emph{$S^*$-absolute} if the topology of $E$ admits a fundamental family $\mathcal P$ of seminorms such that, for every $p\in\mathcal P$ and every scalar sequence $(a_n)$ satisfying
\[
 \limsup_{n\to\infty}\abs{a_n}^{1/n}<\infty,
\]
the map
\[
 x=\sum_{n=0}^{\infty}x_n
 \longmapsto
 \sum_{n=0}^{\infty}\abs{a_n}p(x_n)
\]
is a continuous seminorm on $E$. We state the abstract theorem in the form needed below.

\begin{theorem}[Dineen--Mujica \cite{DineenMujica}]\label{thm:DM}
Let $E$ be a locally convex space with an $S^*$-absolute decomposition $E=\sum_{n\geq0}E_n$. Suppose that the topology of $E$ admits a fundamental family $\mathcal Q$ of seminorms satisfying
\[
 q\left(\sum_{n=0}^{\infty}x_n\right)=\sum_{n=0}^{\infty}q(x_n),
 \qquad x_n\in E_n,
\]
and that every $E_n$ has a Schauder basis $(u_{n,m})_{m\geq1}$. For $q\in\mathcal Q$, let $c_{q,n}$ be the least constant such that
\[
 q\left(\sum_{m=1}^{s}a_mu_{n,m}\right)
 \leq c_{q,n}
 q\left(\sum_{m=1}^{t}a_mu_{n,m}\right)
\]
whenever $s<t$ and $a_1,\ldots,a_t\in\C$. If
\[
 \limsup_{n\to\infty}c_{q,n}^{1/n}<\infty
\]
for every $q\in\mathcal Q$, then $(u_{n,m})_{n,m}$, in every compatible ordering, is a Schauder basis of $E$.
\end{theorem}

\begin{remark}\label{rem:DM-indexing}
The theorem in \cite[Theorem~1]{DineenMujica} is indexed by positive integers. The version above, indexed by $\Nzero$, is equivalent: one may adjoin the one-dimensional component $E_0$ with its one-element basis and basis constant equal to $1$. We use only this abstract criterion, not the topology-specific applications developed later in that paper.
\end{remark}

\section{A fundamental system of compact solid sets}\label{sec:compact}

A Schauder basis $(e_j)$ is called $1$-unconditional if
\[
 \norm{\sum_{j=1}^{\infty}\lambda_jx_je_j}
 \leq
 \norm{\sum_{j=1}^{\infty}x_je_j}
\]
whenever $\abs{\lambda_j}\leq1$ for all $j$. In particular, every coordinate projection and every diagonal operator with unimodular diagonal entries is a contraction.

A subset $A\subset X$ is called \emph{solid} with respect to $(e_j)$ if, whenever
\[
 x=\sum_{j=1}^{\infty}x_je_j\in A
 \quad\text{and}\quad
 y=\sum_{j=1}^{\infty}y_je_j
 \quad\text{satisfy}\quad
 \abs{y_j}\leq\abs{x_j}\ \text{for all }j,
\]
then $y\in A$.

Let
\[
 c_0^+(\Nzero)
 =\bigl\{\eta=(\eta_N)_{N\geq0}:\eta_N>0\text{ for every }N
 \text{ and }\eta_N\longrightarrow0\bigr\}.
\]
For $\eta\in c_0^+(\Nzero)$, define
\begin{equation}\label{eq:Keta}
 K_\eta
 =\bigl\{x\in X:\norm{(\id-\pi_N)x}\leq\eta_N
 \text{ for every }N\in\Nzero\bigr\}.
\end{equation}
The condition for $N=0$ gives $\norm{x}\leq\eta_0$.

\begin{lemma}\label{lem:uniform-tail}
For every compact subset $K\subset X$,
\[
 \sup_{x\in K}\norm{(\id-\pi_N)x}\longrightarrow0.
\]
\end{lemma}

\begin{proof}
The operators $\id-\pi_N$ are uniformly bounded and converge pointwise to zero. Let $M=\sup_N\norm{\id-\pi_N}<\infty$. Given $\varepsilon>0$, choose a finite $\varepsilon/(3M)$-net $\{x_1,\ldots,x_r\}$ in $K$. Choose $N_0$ such that
\[
 \norm{(\id-\pi_N)x_j}<\frac{\varepsilon}{3}
\]
for $N\geq N_0$ and $1\leq j\leq r$. If $x\in K$, choose $j$ with $\norm{x-x_j}<\varepsilon/(3M)$. Then, for $N\geq N_0$,
\[
 \norm{(\id-\pi_N)x}
 \leq M\norm{x-x_j}+\norm{(\id-\pi_N)x_j}
 <\frac{2\varepsilon}{3}<\varepsilon.
\]
\end{proof}

\begin{proposition}\label{prop:fundamental-compacts}
Let $X$ have a $1$-unconditional Schauder basis. Then the family
\[
 \{K_\eta:\eta\in c_0^+(\Nzero)\}
\]
is a fundamental system of compact, absolutely convex, solid subsets of $X$. Moreover, for every $r>0$,
\begin{equation}\label{eq:scaling}
 rK_\eta=K_{r\eta}.
\end{equation}
\end{proposition}

\begin{proof}
Fix $\eta\in c_0^+(\Nzero)$. Each set
\[
 \bigl\{x:\norm{(\id-\pi_N)x}\leq\eta_N\bigr\}
\]
is closed and absolutely convex, so $K_\eta$ is closed and absolutely convex. It is bounded because $\norm{x}\leq\eta_0$ for $x\in K_\eta$.

To prove total boundedness, let $\varepsilon>0$ and choose $N$ such that $\eta_N<\varepsilon/2$. The set $\pi_N(K_\eta)$ is bounded in the finite-dimensional space $\pi_N(X)$ and hence is totally bounded. Choose $y_1,\ldots,y_r\in\pi_N(X)$ such that
\[
 \pi_N(K_\eta)\subset\bigcup_{j=1}^{r}B(y_j,\varepsilon/2).
\]
For $x\in K_\eta$, choose $j$ with $\norm{\pi_Nx-y_j}<\varepsilon/2$. Then
\[
 \norm{x-y_j}
 \leq\norm{(\id-\pi_N)x}+\norm{\pi_Nx-y_j}
 <\varepsilon.
\]
Thus $K_\eta$ is totally bounded. Since it is closed in the Banach space $X$, it is compact.

Let $K\subset X$ be compact and define
\[
 \delta_N=\sup_{x\in K}\norm{(\id-\pi_N)x},\qquad N\in\Nzero.
\]
By Lemma~\ref{lem:uniform-tail}, $\delta_N\to0$. Hence
\[
 \eta_N=\delta_N+\frac{1}{N+1}
\]
defines an element of $c_0^+(\Nzero)$, and $K\subset K_\eta$. Therefore the family is fundamental.

Suppose now that $x\in K_\eta$ and $\abs{y_j}\leq\abs{x_j}$ for every $j$. Since the basis is $1$-unconditional, coordinatewise domination gives
\[
 \norm{(\id-\pi_N)y}
 \leq\norm{(\id-\pi_N)x}
 \leq\eta_N
\]
for every $N$. Thus $y\in K_\eta$, proving solidity. Finally, \eqref{eq:scaling} follows directly from the definition.
\end{proof}

\section{Fourier projections and the basis constants}\label{sec:fourier}

Let $K\subset X$ be a bounded solid set. For $j\in\N$ and $\theta\in\mathbb R$, define the coordinate rotation
\[
 R_{j,\theta}\left(\sum_{i=1}^{\infty}x_ie_i\right)
 =\sum_{i\neq j}x_ie_i+\e^{i\theta}x_je_j.
\]
Solidity implies $R_{j,\theta}(K)=K$. For $r\in\Nzero$ and $P\in\calP({}^nX)$, define
\begin{equation}\label{eq:Delta}
 \Delta_{j,r}P(x)
 =\frac{1}{2\pi}\int_0^{2\pi}
 P(R_{j,\theta}x)\e^{-ir\theta}\dd\theta.
\end{equation}
This operator selects the monomials for which the exponent of the $j$th coordinate is equal to $r$.

\begin{lemma}\label{lem:fourier-contraction}
If $K$ is bounded and solid, then
\[
 \norm{\Delta_{j,r}P}_K\leq\norm{P}_K
\]
for every $j,r$ and every homogeneous polynomial $P$. Every finite composition of Fourier projections involving distinct coordinates is also a contraction for $\norm{\cdot}_K$.
\end{lemma}

\begin{proof}
For $x\in K$, the rotation $R_{j,\theta}x$ belongs to $K$. Therefore
\[
 \abs{\Delta_{j,r}P(x)}
 \leq\frac{1}{2\pi}\int_0^{2\pi}
 \abs{P(R_{j,\theta}x)}\dd\theta
 \leq\norm{P}_K.
\]
Now let $j_1,\ldots,j_s$ be distinct coordinates and let $r_1,\ldots,r_s\in\Nzero$. Since the coordinate rotations commute, repeated use of \eqref{eq:Delta} and Fubini's theorem gives
\begin{align*}
 &\bigl(\Delta_{j_1,r_1}\cdots\Delta_{j_s,r_s}P\bigr)(x)\\
 &\quad=\frac{1}{(2\pi)^s}
 \int_{[0,2\pi]^s}
 P\bigl(R_{j_1,\theta_1}\cdots R_{j_s,\theta_s}x\bigr)
 \e^{-i(r_1\theta_1+\cdots+r_s\theta_s)}
 \dd\theta_1\cdots\dd\theta_s.
\end{align*}
Solidity implies that every point $R_{j_1,\theta_1}\cdots R_{j_s,\theta_s}x$ belongs to $K$. Taking absolute values in the preceding formula and then the supremum over $x\in K$ yields
\[
 \norm{\Delta_{j_1,r_1}\cdots\Delta_{j_s,r_s}P}_K
 \leq\norm{P}_K.
\]
\end{proof}

Let $P\in\calP({}^nX)$ and let $\gamma$ be a degree-$n$ multi-index, with $k=\ellength(\gamma)$. The finite-dimensional polynomial $P\circ\pi_k$ has a unique expansion
\[
 P\circ\pi_k
 =\sum_{\substack{\abs{\beta}=n\\ \ellength(\beta)\leq k}}
 a_\beta(P)e^{*\beta}.
\]
We define the square partial sum ending at $e^{*\gamma}$ by
\begin{equation}\label{eq:Sgamma-definition}
 S_\gamma P
 =\sum_{\substack{\abs{\beta}=n\\ \beta\preceq\gamma}}
 a_\beta(P)e^{*\beta}.
\end{equation}
This definition is independent of the chosen finite-dimensional coordinate space containing the support of $\gamma$, because finite-dimensional monomial coefficients are unique.

For a bounded solid set $K\subset X$, let $c_{K,n}\in[0,\infty]$ be the least constant $C$ such that
\begin{equation}\label{eq:cKn-definition}
 \norm{\sum_{m=1}^{s}a_mP_{n,m}}_K
 \leq C\norm{\sum_{m=1}^{t}a_mP_{n,m}}_K
\end{equation}
whenever $s<t$ and $a_1,\ldots,a_t\in\C$. Once the monomials are known to be a Schauder basis, $c_{K,n}$ is their usual basis constant relative to $\norm{\cdot}_K$.

\begin{proposition}\label{prop:n+1}
Let $X$ have a $1$-unconditional Schauder basis and let $K\subset X$ be bounded and solid. For every $n\in\Nzero$, every $P\in\calP({}^nX)$ and every degree-$n$ multi-index $\gamma$,
\begin{equation}\label{eq:partial-bound}
 \norm{S_\gamma P}_K\leq(n+1)\norm{P}_K.
\end{equation}
Consequently, the constant $c_{K,n}$ defined by \eqref{eq:cKn-definition} satisfies
\begin{equation}\label{eq:basis-constant}
 c_{K,n}\leq n+1.
\end{equation}
\end{proposition}

\begin{proof}
The case $n=0$ is immediate. Assume $n\geq1$, write $k=\ellength(\gamma)$, and set
\[
 Q=P\circ\pi_k.
\]
Since $\pi_k(K)\subset K$, one has $\norm{Q}_K\leq\norm{P}_K$. Moreover, every multi-index $\beta\preceq\gamma$ satisfies $\ellength(\beta)\leq k$. Thus every monomial occurring in $S_\gamma P$ depends only on the first $k$ coordinates, and the corresponding coefficients of $P$ and $P\circ\pi_k$ coincide. Consequently,
\begin{equation}\label{eq:truncate-partial-sum}
 S_\gamma P=S_\gamma(P\circ\pi_k)=S_\gamma Q.
\end{equation}

For $1\leq j\leq k$ and $0\leq r<\gamma_j$, define
\[
 T_{j,r}^{\gamma}
 =\Delta_{j,r}\Delta_{j+1,\gamma_{j+1}}\cdots\Delta_{k,\gamma_k},
\]
where the product after $\Delta_{j,r}$ is omitted when $j=k$. Also set
\[
 T_\gamma=\Delta_{1,\gamma_1}\cdots\Delta_{k,\gamma_k}.
\]
We claim that
\begin{equation}\label{eq:partition}
 S_\gamma P
 =\sum_{j=1}^{k}\sum_{r=0}^{\gamma_j-1}T_{j,r}^{\gamma}Q
 +T_\gamma Q.
\end{equation}

To verify this identity, consider the finite initial set
\[
 \mathcal I_\gamma
 =\{\beta\in\Nzero^{(\N)}:\abs{\beta}=n,
 \ \beta\preceq\gamma\}.
\]
It admits the disjoint decomposition
\begin{equation}\label{eq:index-partition}
 \mathcal I_\gamma
 =\{\gamma\}
 \mathbin{\dot\cup}
 \bigcup_{j=1}^{k}\ \bigcup_{r=0}^{\gamma_j-1}
 \bigl\{\beta:\abs{\beta}=n,
 \ \beta_i=\gamma_i\ (i>j),\ \beta_j=r\bigr\}.
\end{equation}
Indeed, if $\beta\prec\gamma$, let
\[
 j=\max\{i:\beta_i\neq\gamma_i\}.
\]
By the definition of the square order, $\beta_j<\gamma_j$ and $\beta_i=\gamma_i$ for every $i>j$, so $\beta$ belongs to exactly one set on the right-hand side of \eqref{eq:index-partition}. Conversely, every multi-index in one of those sets precedes $\gamma$. The operator $T_\gamma$ selects the monomial indexed by $\gamma$, whereas $T_{j,r}^{\gamma}$ selects precisely the monomials in the corresponding set of \eqref{eq:index-partition}. Together with \eqref{eq:truncate-partial-sum}, this proves \eqref{eq:partition}.

The number of terms on the right-hand side of \eqref{eq:partition} is
\[
 1+\sum_{j=1}^{k}\gamma_j=n+1.
\]
By Lemma~\ref{lem:fourier-contraction}, each operator in \eqref{eq:partition} is a contraction for $\norm{\cdot}_K$. Thus
\[
 \norm{S_\gamma P}_K
 \leq(n+1)\norm{Q}_K
 \leq(n+1)\norm{P}_K,
\]
which proves \eqref{eq:partial-bound}. To obtain the basis-constant estimate explicitly, let $s<t$ and put
\[
 P=\sum_{m=1}^{t}a_mP_{n,m}.
\]
If $P_{n,s}=e^{*\gamma}$, then
\[
 S_\gamma P=\sum_{m=1}^{s}a_mP_{n,m}.
\]
Hence \eqref{eq:partial-bound} gives
\[
 \norm{\sum_{m=1}^{s}a_mP_{n,m}}_K
 \leq(n+1)
 \norm{\sum_{m=1}^{t}a_mP_{n,m}}_K,
\]
and therefore $c_{K,n}\leq n+1$.
\end{proof}

\begin{corollary}\label{cor:fixed-degree}
For every $n\in\Nzero$, the degree-$n$ monomials in the square order form a Schauder basis of $(\calP({}^nX),\tau_0)$.
\end{corollary}

\begin{proof}
The case $n=0$ is immediate, since the only degree-zero monomial is the constant function $1$. Assume $n\geq1$. Fix a compact set $K\subset X$. By Proposition~\ref{prop:fundamental-compacts}, $K$ is contained in some compact solid set $K_\eta$, so it is enough to prove convergence uniformly on $K_\eta$.

Let $P\in\calP({}^nX)$ and let $\check P$ be its associated continuous symmetric $n$-linear form. Since $K_\eta$ is bounded and $\pi_N\to\id$ uniformly on $K_\eta$, the identity
\[
 P(x)-P(y)
 =\sum_{j=1}^{n}
 \check P(\underbrace{x,\ldots,x}_{n-j},x-y,
          \underbrace{y,\ldots,y}_{j-1})
\]
shows that
\[
 \norm{P-P\circ\pi_N}_{K_\eta}\longrightarrow0.
\]
The polynomial $P\circ\pi_N$ depends on only finitely many coordinates and is therefore a finite linear combination of degree-$n$ monomials. If $S_m$ denotes the $m$th square partial-sum operator, then for all sufficiently large $m$,
\[
 S_m(P\circ\pi_N)=P\circ\pi_N.
\]
Let $\varepsilon>0$. Choose $N$ so large that
\[
 (n+2)\norm{P-P\circ\pi_N}_{K_\eta}<\varepsilon.
\]
Then choose $m_0$ such that $S_m(P\circ\pi_N)=P\circ\pi_N$ for every $m\geq m_0$. By Proposition~\ref{prop:n+1}, for $m\geq m_0$,
\begin{align*}
 \norm{S_mP-P}_{K_\eta}
 &\leq \norm{S_m(P-P\circ\pi_N)}_{K_\eta}
       +\norm{P\circ\pi_N-P}_{K_\eta}\\
 &\leq(n+2)\norm{P-P\circ\pi_N}_{K_\eta}
 <\varepsilon.
\end{align*}
Hence the monomial partial sums converge to $P$ in $\tau_0$. To see that the coefficient functionals are continuous, fix a degree-$n$ multi-index $\alpha$ with $k=\ellength(\alpha)$ and positive numbers $r_1,\ldots,r_k$. The set
\[
 D_r=\left\{\sum_{j=1}^{k}z_je_j:\abs{z_j}\leq r_j,
 \ 1\leq j\leq k\right\}
\]
is compact. The finite-dimensional Cauchy formula gives, for the coefficient $a_\alpha(P)$ of $e^{*\alpha}$,
\[
 \abs{a_\alpha(P)}
 \leq \frac{\norm{P}_{D_r}}{r_1^{\alpha_1}\cdots r_k^{\alpha_k}}.
\]
Thus every monomial coefficient functional is $\tau_0$-continuous. Uniqueness follows by restricting $P$ to the finite-dimensional spaces $\pi_N(X)$.
\end{proof}

\section{The monomial basis of \texorpdfstring{$\calH(X)$}{H(X)}}\label{sec:main}

From now on, $X$ has a $1$-unconditional Schauder basis. Let $f\in\calH(X)$ have Taylor expansion $f=\sum_{n=0}^{\infty}P_n$. For $\eta\in c_0^+(\Nzero)$, define
\begin{equation}\label{eq:qeta}
 q_\eta(f)=\sum_{n=0}^{\infty}\norm{P_n}_{K_\eta}.
\end{equation}

\begin{proposition}\label{prop:qeta}
The family $(q_\eta)_{\eta\in c_0^+(\Nzero)}$ is a fundamental system of finite-valued seminorms for $(\calH(X),\tau_0)$. Moreover, the Taylor expansion gives the unique homogeneous decomposition
\[
 \calH(X)=\sum_{n=0}^{\infty}\calP({}^nX).
\]
This decomposition is $S^*$-absolute with respect to these seminorms, and
\[
 q_\eta\left(\sum_{n=0}^{\infty}P_n\right)
 =\sum_{n=0}^{\infty}q_\eta(P_n).
\]
\end{proposition}

\begin{proof}
Fix $\eta\in c_0^+(\Nzero)$ and $\rho>1$. By \eqref{eq:scaling}, $\rho K_\eta=K_{\rho\eta}$. For $x\in K_\eta$, the one-variable function $\lambda\mapsto f(\lambda x)$ is entire and its $n$th Taylor coefficient is $P_n(x)$. Cauchy's estimate on the circle $\abs{\lambda}=\rho$ gives
\[
 \abs{P_n(x)}
 \leq \rho^{-n}\norm{f}_{K_{\rho\eta}}.
\]
Therefore the series defining $q_\eta(f)$ converges and
\begin{equation}\label{eq:q-upper}
 q_\eta(f)
 \leq\sum_{n=0}^{\infty}\rho^{-n}\norm{f}_{K_{\rho\eta}}
 =\frac{\rho}{\rho-1}\norm{f}_{K_{\rho\eta}}.
\end{equation}
On the other hand, uniform convergence of the Taylor series on $K_\eta$ yields
\begin{equation}\label{eq:q-lower}
 \norm{f}_{K_\eta}\leq q_\eta(f).
\end{equation}
Since the sets $K_\eta$ form a fundamental system of compact subsets, \eqref{eq:q-upper} and \eqref{eq:q-lower} show that the seminorms $q_\eta$ generate $\tau_0$.

The uniqueness of the homogeneous decomposition is the uniqueness of the Taylor expansion at the origin, and the displayed additivity is immediate from the definition. To verify the $S^*$-absolute property, let $(a_n)$ be a scalar sequence such that
\[
 L=\limsup_{n\to\infty}\abs{a_n}^{1/n}<\infty.
\]
Choose numbers $A$ and $\rho$ with $L<A<\rho$. There is $C>0$ such that $\abs{a_n}\leq CA^n$ for every $n$. Define
\[
 q_{\eta,a}(f)=\sum_{n=0}^{\infty}\abs{a_n}\norm{P_n}_{K_\eta}.
\]
By the same Cauchy estimate,
\begin{align*}
 q_{\eta,a}(f)
 &\leq C\sum_{n=0}^{\infty}\left(\frac{A}{\rho}\right)^n
       \norm{f}_{K_{\rho\eta}}\\
 &=\frac{C}{1-A/\rho}\norm{f}_{K_{\rho\eta}}.
\end{align*}
Thus $q_{\eta,a}$ is a finite-valued $\tau_0$-continuous seminorm, proving that the homogeneous decomposition is $S^*$-absolute.
\end{proof}

\begin{theorem}\label{thm:main-1unconditional}
Let $X$ be a complex Banach space with a $1$-unconditional Schauder basis. Then the associated monomials, with the square order in each homogeneous degree and with any compatible global ordering, form a Schauder basis of $(\calH(X),\tau_0)$.
\end{theorem}

\begin{proof}
By Corollary~\ref{cor:fixed-degree}, the square-ordered monomials form a Schauder basis in every homogeneous component $\calP({}^nX)$. Proposition~\ref{prop:qeta} provides an $S^*$-absolute homogeneous decomposition and a fundamental family of additive seminorms. On $\calP({}^nX)$, the restriction of $q_\eta$ is exactly $\norm{\cdot}_{K_\eta}$. Hence Proposition~\ref{prop:n+1} gives
\[
 c_{q_\eta,n}=c_{K_\eta,n}\leq n+1.
\]
Consequently,
\[
 \limsup_{n\to\infty}c_{q_\eta,n}^{1/n}
 \leq\lim_{n\to\infty}(n+1)^{1/n}=1.
\]
All hypotheses of the abstract criterion \cite[Theorem~1]{DineenMujica} are therefore satisfied, and the conclusion follows from Theorem~\ref{thm:DM}.
\end{proof}

\begin{corollary}\label{cor:unconditional-renorming}
Let $X$ be a complex Banach space with an unconditional Schauder basis. Then the monomials form a Schauder basis of $(\calH(X),\tau_0)$ in every compatible ordering.
\end{corollary}

\begin{proof}
An unconditional basis admits an equivalent norm for which it is $1$-unconditional; see, for example, \cite[Proposition~3.1.3]{AlbiacKalton}. Equivalent norms define the same compact subsets, the same holomorphic mappings and the same compact-open topology. Theorem~\ref{thm:main-1unconditional} therefore applies.
\end{proof}

\section{Applications and examples}\label{sec:applications}

The abstract form of Theorem~\ref{thm:main-1unconditional} applies to a large class of classical and nonclassical Banach spaces. We first record a general consequence for sequence lattices.

\begin{corollary}\label{cor:sequence-lattices}
Let $E$ be a complex Banach space of scalar sequences such that $c_{00}\subset E$ is dense and such that coordinatewise domination implies norm domination: whenever $x=(x_j)\in E$ and $y=(y_j)$ satisfies
\[
 \abs{y_j}\leq \abs{x_j}\qquad(j\in\N),
\]
then $y\in E$ and $\norm{y}_E\leq\norm{x}_E$. Then the canonical unit vectors form a $1$-unconditional Schauder basis of $E$. Consequently, the associated monomials, with the square order in each homogeneous degree and any compatible global ordering, form a Schauder basis of
\[
 (\calH(E),\tau_0).
\]
\end{corollary}

\begin{proof}
For $A\subset\N$, let
\[
 P_Ax=\sum_{j\in A}x_je_j.
\]
Since $\abs{(P_Ax)_j}\leq\abs{x_j}$ for every $j$, the lattice property gives
\[
 \norm{P_Ax}_E\leq\norm{x}_E.
\]
More generally, if $(\lambda_j)\in\ell_\infty$ and $\abs{\lambda_j}\leq1$ for all $j$, then
\[
 \norm{(\lambda_jx_j)_j}_E\leq\norm{x}_E.
\]
Thus the canonical vectors are $1$-unconditional.

Let $\pi_N=P_{\{1,\ldots,N\}}$. Given $x\in E$ and $\varepsilon>0$, choose $y\in c_{00}$ with $\norm{x-y}_E<\varepsilon$. For all sufficiently large $N$, $\pi_Ny=y$, and hence
\[
 \norm{x-\pi_Nx}_E
 =\norm{(\id-\pi_N)(x-y)}_E
 \leq\norm{x-y}_E
 <\varepsilon.
\]
Therefore the canonical vectors form a Schauder basis. The conclusion follows from Theorem~\ref{thm:main-1unconditional}.
\end{proof}

\begin{example}\label{ex:sequence-lattice-examples}
Corollary~\ref{cor:sequence-lattices} contains, among others, the following familiar families.
\begin{enumerate}
 \item The classical spaces $c_0$ and $\ell_p$, $1\leq p<\infty$.

 \item Weighted sequence spaces
 \[
  \ell_p(v)=\left\{x:\sum_{j=1}^{\infty}\abs{x_j}^pv_j<\infty\right\},
  \qquad 1\leq p<\infty,
 \]
 for positive weights $v=(v_j)$, and the corresponding weighted $c_0$ spaces
 \[
  c_0(v)=\{x:v_j\abs{x_j}\to0\},
  \qquad \norm{x}_{c_0(v)}=\sup_jv_j\abs{x_j}.
 \]

 \item Lorentz sequence spaces $d(w,p)$, $1\leq p<\infty$, endowed with
 \[
  \norm{x}_{d(w,p)}
  =\left(\sum_{j=1}^{\infty}(x_j^*)^pw_j\right)^{1/p},
 \]
 under the usual assumptions that $w=(w_j)$ is decreasing, positive, $w_1=1$, and $\sum_jw_j=\infty$. Their canonical basis is in fact $1$-symmetric.

 \item Orlicz hearts $h_M$, namely the closure of $c_{00}$ in the Orlicz sequence space associated with an Orlicz function $M$, equipped with its Luxemburg norm. The lattice property of that norm and the definition of $h_M$ place these spaces directly under Corollary~\ref{cor:sequence-lattices}.

 \item Schreier-type sequence spaces. If $\mathcal F$ is a hereditary family of finite subsets of $\N$ containing all singletons and $X_{\mathcal F}$ denotes the completion of $c_{00}$ for
 \[
  \norm{x}_{\mathcal F}
  =\sup_{F\in\mathcal F}\sum_{j\in F}\abs{x_j},
 \]
 then the canonical basis is $1$-unconditional. In particular, this applies to the classical Schreier space obtained from
 \[
  \mathcal S=\{F\subset\N:F\text{ finite and }\abs{F}\leq\min F\}.
 \]
\end{enumerate}
Thus each of these spaces carries the monomial Schauder basis furnished by Theorem~\ref{thm:main-1unconditional}.
\end{example}

The preceding examples can be enlarged further by passing to vector-valued sequence spaces.

\begin{proposition}\label{prop:E-sums}
Let $E$ satisfy the hypotheses of Corollary~\ref{cor:sequence-lattices}, and let $(X_j)_{j\geq1}$ be a sequence of complex Banach spaces. Assume that each $X_j$ has a $1$-unconditional Schauder basis $(e_{j,k})_{k\geq1}$. Define
\[
 X=\left(\bigoplus_{j=1}^{\infty}X_j\right)_E
 =\left\{(x_j):x_j\in X_j,\ (\norm{x_j}_{X_j})_{j\geq1}\in E\right\}
\]
with norm
\[
 \norm{(x_j)}_X
 =\norm{(\norm{x_j}_{X_j})_{j\geq1}}_E.
\]
Then the double family $(e_{j,k})_{j,k}$, viewed in the corresponding coordinate blocks of $X$, is a $1$-unconditional Schauder basis of $X$ under any enumeration of the double index set. Consequently, the associated monomials form a Schauder basis of
\[
 (\calH(X),\tau_0)
\]
for every compatible ordering.
\end{proposition}

\begin{proof}
Let $F\subset\N\times\N$ be finite, and let $P_F$ denote the corresponding coordinate projection. If $x=(x_j)\in X$ and $x_j=\sum_{k=1}^{\infty}a_{j,k}e_{j,k}$,
then, by the $1$-unconditionality in each $X_j$,
\[
 \norm{(P_Fx)_j}_{X_j}\leq\norm{x_j}_{X_j}
 \qquad(j\in\N).
\]
The lattice property of $E$ therefore gives
\[
 \norm{P_Fx}_X
 =\norm{(\norm{(P_Fx)_j}_{X_j})_j}_E
 \leq\norm{(\norm{x_j}_{X_j})_j}_E
 =\norm{x}_X.
\]
Hence all finite coordinate projections are contractions.

It remains to verify density of the finite linear span of the vectors $e_{j,k}$. Given $x=(x_j)\in X$ and $\varepsilon>0$, first choose $J$ so that
\[
 \norm{(0,\ldots,0,\norm{x_{J+1}},\norm{x_{J+2}},\ldots)}_E<\frac{\varepsilon}{2},
\]
which is possible because the canonical basis of $E$ is a Schauder basis. For each $1\leq j\leq J$, choose a finite partial sum $y_j$ of the basis expansion of $x_j$ so close to $x_j$ that
\[
 \norm{(\norm{x_1-y_1},\ldots,\norm{x_J-y_J},0,0,\ldots)}_E<\frac{\varepsilon}{2}.
\]
Then $y=(y_1,\ldots,y_J,0,\ldots)$ is a finite linear combination of the vectors $e_{j,k}$ and $\norm{x-y}_X<\varepsilon$. Thus the double family is total. Since its finite coordinate projections are contractions, it is a $1$-unconditional Schauder basis; in particular, every enumeration is a Schauder basis. The last assertion follows from Theorem~\ref{thm:main-1unconditional}.
\end{proof}

\begin{corollary}\label{cor:mixed-sums}
Let $1\leq r<\infty$, let $(p_j)_{j\geq1}\subset[1,\infty)$, and let $(d_j)_{j\geq1}\subset\N$. Then the conclusion of Theorem~\ref{thm:main-1unconditional} holds for
\[
 c_0\left(\bigoplus_{j=1}^{\infty}\ell_{p_j}^{d_j}\right)
 \qquad\text{and}\qquad
 \ell_r\left(\bigoplus_{j=1}^{\infty}\ell_{p_j}^{d_j}\right).
\]
More generally, the same conclusion holds for every $E$-sum in Proposition~\ref{prop:E-sums}.
\end{corollary}

\begin{proof}
Apply Proposition~\ref{prop:E-sums} with $E=c_0$ or $E=\ell_r$ and $X_j=\ell_{p_j}^{d_j}$.
\end{proof}

We now return to the two spaces that originally motivated the present work. The first one is recovered as a particularly simple instance of the preceding vector-valued construction.

\begin{corollary}\label{cor:c0sum}
Let $1\leq p<\infty$ and
\[
 X_p=c_0\left(\bigoplus_{j=1}^{\infty}\ell_p^j\right)
 =\left\{x=(x^{(j)})_{j\geq1}:
 x^{(j)}\in\ell_p^j,
 \ \norm{x^{(j)}}_p\longrightarrow0\right\},
\]
endowed with the norm $\norm{x}=\sup_j\norm{x^{(j)}}_p$. Then the monomials associated with the canonical basis of $X_p$, with any compatible ordering, form a Schauder basis of $(\calH(X_p),\tau_0)$.
\end{corollary}

\begin{proof}
Apply Proposition~\ref{prop:E-sums} with $E=c_0$ and $X_j=\ell_p^j$ for every $j\geq1$.
\end{proof}

Let $w=(w_j)$ be a decreasing sequence of positive numbers such that $w_1=1$, $w\in c_0\setminus\ell_1$, and put
$W_m=\sum_{j=1}^{m}w_j$.

For a scalar sequence $x$, let $(x_j^*)$ denote the decreasing rearrangement of $(\abs{x_j})$. The Lorentz predual is
\[
 d_*(w,1)
 =\left\{x\in c_0:
 \lim_{m\to\infty}\frac{\sum_{j=1}^{m}x_j^*}{W_m}=0\right\},
\]
with norm
\[
 \norm{x}_*
 =\sup_{m\geq1}\frac{\sum_{j=1}^{m}x_j^*}{W_m}.
\]

\begin{remark}\label{rem:lorentz-warning}
A rearrangement-based condition alone does not provide compactness in $d_*(w,1)$. Indeed, let $\lambda_m=1/W_m$. Since $w\notin\ell_1$, one has $W_m\to\infty$ and therefore $\lambda_m\to0$. Nevertheless, the symmetric set
\[
 A_\lambda
 =\left\{x\in d_*(w,1):
 \frac{\sum_{i=1}^{m}x_i^*}{W_m}\leq\lambda_m
 \text{ for every }m\right\}
\]
contains every unit vector $e_j$. Moreover, for $j\neq k$, $\norm{e_j-e_k}_*
 \geq \frac{(e_j-e_k)_1^*}{W_1}=1$.

Thus $A_\lambda$ is not totally bounded and hence is not relatively compact. This explains why the coordinate-tail compact sets $K_\eta$ are needed in the Lorentz-predual case.
\end{remark}

\begin{lemma}\label{lem:dstar-basis}
The canonical unit vectors form a normalized $1$-symmetric Schauder basis of $d_*(w,1)$.
\end{lemma}

\begin{proof}
The norm and the defining limit depend only on the decreasing rearrangement, so permutations and coordinatewise multiplication by unimodular scalars are isometries. Coordinatewise domination is contractive. It remains to prove convergence of the coordinate partial sums.

Fix $x\in d_*(w,1)$ and $\varepsilon>0$. Choose $M$ such that $\frac{\sum_{j=1}^{m}x_j^*}{W_m}<\varepsilon$ for all $m\geq M$.

Since $x\in c_0$, choose $N$ such that $\sup_{j>N}\abs{x_j}<\frac{\varepsilon}{M}$.
Put $y=x-\pi_Nx$. If $m<M$, then $W_m\geq W_1=1$ and $\frac{\sum_{j=1}^{m}y_j^*}{W_m}
 \leq m\sup_{j>N}\abs{x_j}<\varepsilon$.

If $m\geq M$, coordinatewise domination gives $\sum_{j=1}^{m}y_j^*\leq\sum_{j=1}^{m}x_j^*$,
and hence the same quotient is less than $\varepsilon$. Therefore $\norm{x-\pi_Nx}_*<\varepsilon$.
Thus the canonical vectors form a Schauder basis, and the preceding symmetry observation shows that it is normalized and $1$-symmetric.
\end{proof}

\begin{corollary}\label{cor:dstar}
With any compatible ordering, the monomials associated with the canonical basis of $d_*(w,1)$ form a Schauder basis of
$(\calH(d_*(w,1)),\tau_0)$.
\end{corollary}

\begin{proof}
Apply Theorem~\ref{thm:main-1unconditional} and Lemma~\ref{lem:dstar-basis}.
\end{proof}

\begin{remark}
The proof gives the explicit subexponential estimate
\[
 c_{K_\eta,n}\leq n+1
\]
for every $\eta$ and every degree $n$. In particular, no auxiliary constant depending on a product estimate is required, and the same argument applies simultaneously to the $c_0$-sum and Lorentz-predual cases.
\end{remark}

\end{document}